\documentclass[letterpaper, 10 pt, conference]{ieeeconf}  
\IEEEoverridecommandlockouts                              
\let\labelindent\relax

\usepackage{amsmath} 
\usepackage{amssymb}  
\usepackage{amsfonts}

\newtheorem{proposition}{Proposition}[section]

\usepackage{graphicx}
\usepackage{float} 

\usepackage[margin = 1in]{geometry}

\usepackage{tikz}
\usetikzlibrary{shapes,arrows,calc,positioning}
\tikzstyle{bigblock} = [draw, fill=blue!20, rectangle, 
    minimum height=6em, minimum width=8em]
\tikzstyle{medblock} = [draw, fill=blue!20, rectangle, 
    minimum height=4em, minimum width=4em]    
\tikzstyle{mux} = [draw, fill=black!20, rectangle, 
    minimum height=5em, minimum width=0.1em]    
\tikzstyle{smallblock} = [draw, fill=blue!20, rectangle, 
    minimum height=2em, minimum width=3em]
    
\tikzstyle{data_block} = [draw, fill=green!20, rectangle, 
    minimum height=2em, minimum width=3em]
\tikzstyle{ops_block} = [draw, fill=blue!20, rectangle, 
    minimum height=2em, minimum width=3em]    
\tikzstyle{est_block} = [draw, fill=red!20, rectangle, 
    minimum height=2em, minimum width=3em]    
    
\tikzstyle{sum} = [draw, fill=blue!20, circle, node distance=1cm,minimum height=0.5cm]
\tikzstyle{signal} = [coordinate]
\tikzstyle{pinstyle} = [pin edge={to-,thin,black}]
\tikzstyle{block} = [draw, fill=blue!20, rectangle, 
    minimum height=3em, minimum width=9em]
\tikzstyle{blockS} = [draw, fill=blue!20, rectangle, 
    minimum height=3em, minimum width=4em]    
\tikzstyle{input} = [coordinate]
\tikzstyle{output} = [coordinate]
\usetikzlibrary{matrix}

\usetikzlibrary{positioning}
\usetikzlibrary{math}

\usepackage{hyperref}
\usepackage{xcolor}
\hypersetup{
    colorlinks,
    linkcolor={blue!100!black},
    citecolor={blue!50!black},
    urlcolor={blue!80!black}
}

\newcommand{\bc}{\begin{center}}
\newcommand{\ec}{\end{center}}
\newcommand{\benum}{\begin{enumerate}}
\newcommand{\eenum}{\end{enumerate}}
\newcommand{\nn}{\nonumber}
\newcommand{\matl}{\left[ \begin{array}}
\newcommand{\matr}{\end{array} \right]}
\newcommand{\matls}{\left[ \begin{smallmatrix}}
\newcommand{\matrs}{\end{smallmatrix} \right]}
\newcommand{\isdef}{\stackrel{\triangle}{=}}

\newcommand{\rmD}{{\rm D}}

\newcommand{\rmI}{{\rm I}}

\newcommand{\rmN}{{\rm N}}

\newcommand{\rmP}{{\rm P}}

\newcommand{\rmT}{{\rm T}}

\newcommand{\rmf}{{\rm f}}

\newcommand{\BBR}{{\mathbb R}}

\newcommand{\shiftq}{{\textbf{\textrm{q}}}}

\renewcommand{\matl}{\begin{bmatrix}}
\renewcommand{\matr}{\end{bmatrix} }

\usepackage[style=ieee ,sorting=none,maxbibnames=99,giveninits, doi=false, backend=biber]{biblatex}
\AtBeginBibliography{\footnotesize}

\title{\LARGE An In-situ Solid Fuel Ramjet Thrust Monitoring and Regulation Framework Using Neural Networks and Adaptive Control}

\author{\large Ryan DeBoskey, Parham Oveissi, Venkat Narayanaswamy, and Ankit Goel
\thanks{Ryan DeBoskey and Venkat Narayanaswamy are with the Department of Mechanical Engineering, North Carolina State University, Raleigh, NC 27695 {\tt\small rddebosk@ncsu.edu, vnaraya3@ncsu.edu }}
\thanks{Parham Oveissi and Ankit Goel are with the Department of Mechanical Engineering, University of Maryland, Baltimore County,1000 Hilltop Circle, Baltimore, MD 21250. {\tt\small parhamo1@umbc.edu, ankgoel@umbc.edu }}
}

\begin{document}

\maketitle

\begin{abstract}

Controlling the complex combustion dynamics within solid fuel ramjets (SFRJs) remains a critical challenge limiting deployment at scale.  
This paper proposes the use of a neural network model to process in-situ measurements for monitoring and regulating SFRJ thrust with a learning-based adaptive controller.
A neural network is trained to estimate thrust from synthetic data generated by a feed-forward quasi-one-dimensional SFRJ model with variable inlet control.
An online learning controller based on retrospective cost optimization is integrated with the quasi-one-dimensional SFRJ model to regulate the thrust.  
Sensitivity studies are conducted on both the neural network and adaptive controller to identify optimal hyperparameters.  
Numerical simulation results indicate that the combined neural network and learning control framework can effectively regulate the thrust produced by the SFRJ model using limited in-situ data.  

\end{abstract}
\section{INTRODUCTION}
\label{sec:introduction}

Solid fuel ramjet (SFRJ) platforms promise greater propulsive performance and capability over traditional rocket engines by ingesting freestream air for combustion with high-density solid fuels and without any complex turbomachinery \cite{Schulte1986}.  
Combustion involves complex interaction between turbulent flame-dynamics, solid fuel pyrolysis, and complex hydrocarbon chemistry.  A review of the progress and challenges in understanding the SFRJ combustion phenomena is compiled in several review articles \cite{KRISHNAN1998219,Gany_2009,Veraar_2022}.  Due to the multi-physics nature of the dynamics, precise control over the combustor equivalence ratio and other performance metrics is extremely challenging limiting deployment at scale.

For SFRJ platforms, direct command over system performance requires modifying the inflow conditions to the combustor subsystem.  The majority of studies have considered the use of bypass air channels to modulate system thrust \cite{natan1993experimental,pelosi2003bypass}. Experimentally, bypass ratio control has shown increased combustor efficiency \cite{evans2023performance}, but requires the addition of bypass piping which increases the complexity of the SFRJ geometry whilst decreasing the theoretical capacity for solid fuel loading. 
Recent work has investigated the potential of using variable inlet design as an alternative means for regulating SFRJ performance \cite{deboskey_morphing}.  
Variable inlet control presents potential advantages by minimizing the actuator hardware footprint and providing a novel mechanism to regulate air ingested for combustion and, thus, thrust.
However, due to the complex dynamic relationship between the inlet geometry and thrust generated, which is sensitive to geometric variations of the SFRJ, as well as variations in flow conditions and dependence on flight conditions, designing a control system to regulate the SFRJ thrust remains a challenging problem. 
The focus of this work is thus the investigation of an adaptive control technique to regulate the thrust generated by an SFRJ.

Previous controllers designed for thrust modulation of SFRJ engines have primarily assumed the generated thrust is available as an input to the control system \cite{goel2015retrospective,goel2018retrospective, oveissi2023learning,oveissi2024,oveissi2025adaptive}.
However, in-flight measurement of generated thrust is extremely challenging, particularly for supersonic and hypersonic vehicle platforms \cite{conners1998full}.  Many previous works have been conducted to characterize SFRJ performance, resulting in the development of empirical relations relating inflow conditions to system performance \cite{Schulte1986,netzer1991burning,veraar2018sustained}.  
Although useful in elucidating principal determinants, continued improvement in the understanding of the native complex multiphysics process, through experiments and numerical simulations \cite{deboskey_facility}, is still necessary to develop highly accurate phenomenological models.  
Alternatively, thrust models for the SFRJ system can be constructed using data-driven and machine learning approaches. An inherent advantage of machine learning is that a thrust model can be developed from any available input set. This enables the creation of a model based entirely on in-situ measurements obtained from a real-world supersonic platform.

This work utilizes artificial neural networks (ANNs) to predict the thrust of a model SFRJ platform using in situ measurements.  Previous studies have utilized neural networks to model and improve the performance of SFRJ and related propulsive platforms in a variety of ways.  Thrust prediction in hybrid-rocket \cite{zavoli_nn} and solid-rocket \cite{ZHANG2024114051} engines has been successfully demonstrated using neural networks trained on both synthetic and experimental data, respectively.  Many studies have considered neural networks in surrogate-based analysis and optimization \cite{QUEIPO20051}, where neural networks are used to model and improve the design of propulsion systems \cite{SHYY200159,Saidyanathan_nn_optimization}.  Analogously, neural networks have also been used to determine model sensitivities for solid fuel combustion to optimize experimental design \cite{BOJKO2025113829}.  Fundamentally, ANNs are known to be universal function approximators \cite{zou2009overview}, making them well-suited to quickly and accurately estimate thrust from limited measurements as considered in this work.  



This work investigates the application of an online, learning-based control design technique called retrospective cost adaptive control (RCAC) \cite{rahman2017retrospective} to regulate the thrust generated by an SFRJ model with variable-geometry inlet as the input and ANN-predicted thrust with in situ measurements. 
RCAC has been recently demonstrated as a viable technique to synthesize an adaptive control system to regulate the thrust and prevent inlet unstart in a liquid-fuel scramjet engine \cite{goel2015retrospective,goel2018retrospective,goel2019output} and to regulate thrust in an SFRJ \cite{oveissi2023learning,oveissi2024,oveissi2025adaptive}.
However, in \cite{oveissi2023learning,oveissi2024,oveissi2025adaptive}, the SFRJ input was assumed to be the heat flux and the measurement was assumed to be direct thrust measurements. 
In contrast, this work considers a realistic input, that is, a variable-geometry inlet and in situ measurements to compute thrust. 

The paper is organized as follows. 
Section \ref{sec:numerical_model} describes the numerical SFRJ model and the ANN used to predict SFRJ thrust, 
Section \ref{sec:control} describes the data-driven, learning-based control system to regulate the thrust, 
Section \ref{sec:results} describes several numerical examples demonstrating the application of the proposed approach to regulate the SFRJ thrust, and 
Section \ref{sec:conclusion} concludes the paper. 

\section{Numerical Methodology}
\label{sec:numerical_model}
This study considers steady cruise operation of a model 140 mm external diameter SFRJ at a flight Mach number of 3.25 and altitude of 30 km.
A constant freestream temperature and pressure are assumed based on the 1976 US Standard Atmosphere model \cite{1976atmosphere}.  Table \ref{cruise_condition} provides details of the selected cruise condition. 

\begin{table}[htbp]
\centering
\caption{Overview of selected cruise condition}\label{cruise_condition}
\begin{tabular}{ c  c  c  c }
\hline \hline
$M$ & $H$ [km] & $P_{t0}$ [Pa] & $T_{t0}$  [K]\\
\hline
3.25 & 30 & 63677 & 748.6  \\
\hline
\end{tabular}
\end{table}

\subsection{SFRJ Model}
\label{sec:SFRJmodel}

This work considers a static quasi-one-dimensional feed-forward SFRJ with variable cowl control, modified based on the model developed by DeBoskey et al. \cite{deboskey_morphing}.  Figure \ref{IntegratedSystemDiagram} shows a schematic of a SFRJ platform with thermodynamic stations labeled.  Freestream air, station ``0'', is ingested and compressed through the inlet and isolator system, station ``1'', and enters the combustor entrance, station ``2''.  The solid fuel is embedded within the combustion chamber in station ``3''; the port radius $r_3$ increases as the fuel is consumed during combustion.  Additional mixing and combustion occurs at the combustor aft-end, station ``4'', before expansion and exhaust through the nozzle, station ``e''. 

\begin{figure}[htbp]
\centering
    \includegraphics[width=1\columnwidth]{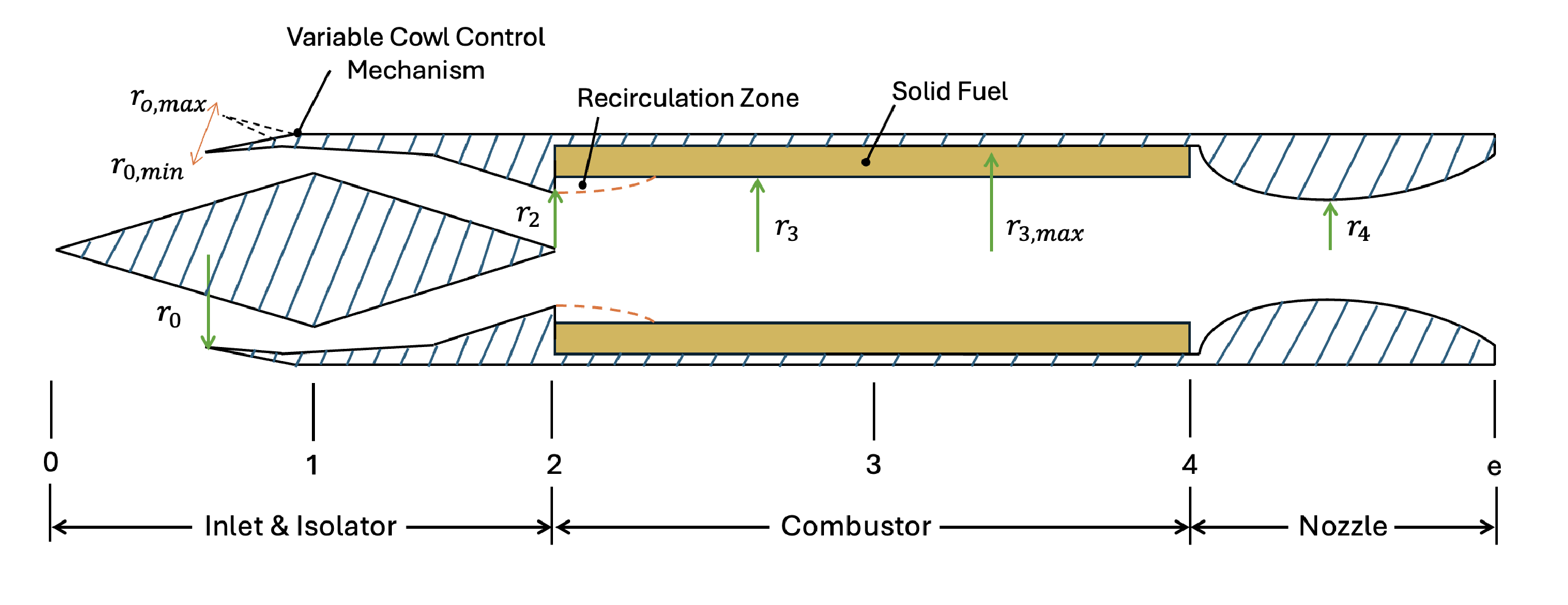}
    \caption{Schematic of model SFRJ projectile with thermodynamic stations labeled.  System performance is controlled through the translation of a variable cowl mechanism.}
    \label{IntegratedSystemDiagram}
\end{figure}

A generic axisymmetric nose-spike center body inlet with a variable geometry cowl is selected for the design. 
The nose-spike and cowl design loosely resembles the NASA 1507 inlet \cite{slater2024wind}. 
The variable geometry mechanism adjusts both the location of the cowl tip in the $x$-direction and the angle of the cowl.
This corresponds to a translation in the cowl tip along the conical shock angle of the center body at the cruise Mach condition. 
The translation is represented in terms of the cowl radius, $r_0$, which is equivalent to the freestream capture area radius.  
This approach is similar to other cowl-based variable geometry methods in which the cowl angle and cowl location are investigated as separate operations \cite{Teng.et.al-CowlLocation,dalle2011performance,Reardon.et.al-CompUnstart,Liu.et.al-Restart,RamprakashMuruganandam-ExpUnstart,TengJianHuacheng-VariableCowl}.  
In \cite{deboskey_morphing}, a series of parametric 2D axisymmetric Reynolds-Averaged Navier-Stokes (RANS) simulations were performed on a structured two-dimensional grid with an increasing freestream capture area.  Isolator exit mass flow rate, mass-averaged stagnation pressure, and mass-averaged Mach number were calculated and curve-fitted as a function of $r_0.$
This model is extended to work across a range of altitudes by normalizing the isolator exit quantities based on freestream conditions.

The isolator exit conditions are then fed forward to a quasi-one-dimensional combustor model.
The regression rate of the solid fuel, $\dot{r}$, is assumed to depend on the inlet mass flux, $G$, total temperature, $T_{t2}$, and combustor pressure, $P_4$, and is given by
\begin{equation}\label{regression}
    \dot{r} =  \alpha G_{a}^{a} P_4^{b} T_{t2}^{c} ,
\end{equation}
\noindent where $\alpha$, $a$, $b$, and $c$ are empirical curve fits from Vaught et al. \cite{vaught1992investigation}.  This work considers a hypothetical model hydroxyl-terminated polybutadiene (HTPB) polymer that is completely converted to the gaseous hydrocarbon, 1,3-butadiene ($C_4 H_6$).  The model assumes that the fuel grain only burns radially, with uniform regression along the axial length of the combustor.
To calculate the combustor aft-end mixing pressure for \eqref{regression}, a simple friction correlation is used to model internal friction, which implies that 
\begin{equation}
    \Delta P_{t,2-4} = \frac{f_D}{4} \frac{L_f}{d} \frac{\rho_3 u_3^2 }{2}  ,
\end{equation}
where $\Delta P_{t,2-4}$ is the pressure change between station ``2'' and station ``4'', $f_D$ is the Darcy friction factor, $L_f$ is the length of the fuel grain, $\rho$ is the port fluid density, and $u_3$ is the port velocity.

Due to the solid fuel regression, the solid fuel port radius, $r_3$, is constantly increasing during SFRJ operation and is updated based on \cite{Hadar1992,evans2023performance} as
\begin{equation}
    r_{3}(t_{i+1}) = r_{3}(t_{i}) - \dot{r}(t_i) (t_{i+1}-t_{i}).
\end{equation}
The simulations are time-integrated until the port radius reaches the maximum allowable radius, $r_{3,\rm max}$. Table \ref{nominalgeo} details the SFRJ geometry considered in this work. 
\begin{table}[htbp] 
 	\centering
 	\caption{Description of SFRJ geometry}
 	\begin{tabular}{ c c c c c}
 		\hline \hline
 		$r_{0}$ [mm] & $r_{2}$ [mm] & $r_3$ &  $r_{t}$ [mm] & $L_{f}$ [mm]\\
 		\hline
 		[47.88,59.28] & 46.7 & [59.2,68.6] & 50.4 & 500\\
            \hline
 	\end{tabular}
 	\label{nominalgeo}
\end{table}

At each time-step, the fuel mass flow rate of 1,3-butadiene, $\dot{m}_f$, is given by
\begin{equation}
    \dot{m}_f 
        =
            A_f \rho_f \dot{r},
\end{equation}
where $A_f = 2 \pi r_3 L_f$ is the exposed surface area of the fuel grain and $\rho_f$ = 900 kg/m$^3$ is the density of HTPB.
Note that a global equivalence ratio, $\phi_G$ can be calculated by
\begin{equation}
    \phi_G = \frac{f}{f_{stoich}} = \frac{\dot{m}_f / \dot{m}_{air}}{f_{stoich}} ,
\end{equation}
where $f_{stoich}$ is the stoichiometric fuel-to-air mass flow ratio.  
The equilibrium flame temperature, $T_{4,\rm eq}$, and the ratio of specific heats, $\gamma_{4}$, at the aft mixing end are calculated using an equilibrium Gibbs solver, assuming constant enthalpy and pressure, in CANTERA \cite{cantera}.
A pressure-comprehensive skeletal kinetics mechanism for 1,3-butadiene combustion developed by Ciottoli et al. \cite{Ciottoli2017} is used to calculate equilibrium products.
Previous work has validated the skeletal mechanism against detailed 1,3-butadiene combustion mechanisms \cite{deboskey2025development,DEBOSKEY2025CNF}.  
The final static temperature at the aft-mixing end of the SFRJ combustor, $T_{4}$, is calculated, assuming a constant combustion efficiency of $\eta_c$ = 0.75, as
\begin{equation}
      T_{4} = \eta_c ( T_{4,\rm eq} - T_{2}) + T_{2}.
\end{equation}


Using isentropic expansion to ambient pressure, the aft-end combustor total pressure and temperature are used to determine a theoretical exhaust velocity, $u_{\rm e,th}$, which is computed as
\begin{align}
    u_{\rm e,th} 
        =
            \big[ 
                2\gamma_4 R_4 T_{t4}(1/(\gamma_4-1)) \cdot \nn \\
                (1 - (P_0 / P_{t4})^{(\gamma_4-1)/\gamma_4})
            \big]^{1/2},
\end{align}
\noindent where $R_4$ and $\gamma_4$ are the specific gas constant and ratio of specific heats, respectively, calculated using the equilibrium composition from station ``4'' in CANTERA.  
This work assumes an idealized nozzle that expands gas at the combustor exit to ambient pressure and operates with a nozzle efficiency, $\eta_n$ = 0.95, to calculate the actual exhaust velocity, $u_{e}$.  
Finally, the generated thrust, $T$, is given by
\begin{equation}
    T = \dot{m}_{air} (1 + f) u_e - \dot{m}_{air} u_0.
\end{equation}

\subsection{Thrust Estimation using Artificial Neural Networks}






This work considers a hypothetical supersonic platform with limited on-board sensing that measures altitude $H$, total combustor pressure $P_{t4}$, and the combustor exhaust composition of carbon monoxide $X_{CO}$. 
Carbon monoxide is chosen due to the extensive study of in-situ combustor measurement through tunable diode laser absorption spectroscopy \cite{wang2000situ,WEBBER2000407}.  
Additionally, feedback of the cowl radius, $r_0$, state is available for thrust estimation.

An ANN is trained to estimate the thrust from the limited on-board sensors.  
The neural network architecture is defined according to a sensitivity study of hyperparameters that evaluates the ANN's performance.
The final model architecture, schematically drawn in Figure \ref{NN_diagram}, is composed of 1 hidden layer composed of 20 neurons.
The input is the available on-board measurements and the output was the estimated thrust.

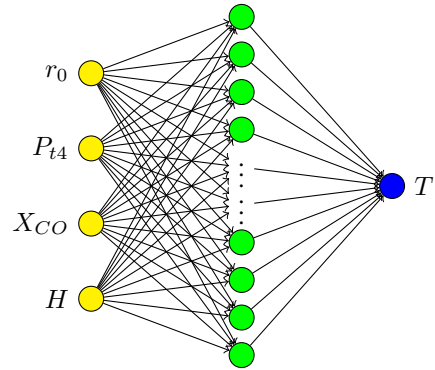
\begin{figure}[htbp]
    \centering
    \begin{tikzpicture}[scale=1, transform shape]
            
            \node[circle, draw, minimum size=0.3cm, fill=yellow] (I-1) at (0,1.5) {};
            \node[anchor=east] at (I-1.west) {$r_0$};
            
            \node[circle, draw, minimum size=0.3cm, fill=yellow] (I-2) at (0,0.5) {};
            \node[anchor=east] at (I-2.west) {$P_{t4}$};

            \node[circle, draw, minimum size=0.3cm, fill=yellow] (I-3) at (0,-0.5) {};
            \node[anchor=east] at (I-3.west) {$X_{CO}$};

            \node[circle, draw, minimum size=0.3cm, fill=yellow] (I-4) at (0,-1.5) {};
            \node[anchor=east] at (I-4.west) {$H$};
            
            \foreach \i in {1,...,4} {
                \node[circle, draw, minimum size=0.3cm, fill=green] (H-\i) at (2,2.75-\i*0.5) {};
            }

            \foreach \i in {5,6} {
                \node (H-\i) at (2,2.75-\i*0.5) {$\vdots$};
            }

            \foreach \i in {7,...,10} {
                \node[circle, draw, minimum size=0.3cm, fill=green] (H-\i) at (2,2.75-\i*0.5) {};
            }
            
            \node[circle, draw, minimum size=0.3cm, fill=blue] (O-1) at (4,0) {};
            \node[anchor=west] at (O-1.east) {$T$};

            \foreach \i in {1,...,4} {
                \foreach \j in {1,...,10} {
                    \draw[->] (I-\i) -- (H-\j);
                }
            }
            
            \foreach \i in {1,...,10} {
                \draw[->] (H-\i) -- (O-1);
            }
            
        \end{tikzpicture}
    \caption{Final artificial neural network architecture diagram (Case NNA)}
    \label{NN_diagram}
\end{figure}

The ANN model was trained using the open-source TensorFlow \cite{abadi2016tensorflow} library.  The ANN was trained across a range of altitude from 10 to 40 km and across the full range of possible $r_0$ and $r_3$ provided in Table \ref{nominalgeo}.  Synthetic thrust measurements were generated using the computational model described in the previous section at 50 discrete points for each independent variable, resulting in a total of 125,000 measurements.  In-situ measurement of combustor pressure $P_{t4}$ and composition of carbon monoxide $X_{CO}$ were collected for each simulation.  The data set was split into a training and testing set, with 80\% and 20\% of the data, respectively.  The ADAM method for stochastic optimization \cite{kingma2014adam} is utilized to minimize the mean-squared error, $MSE$, loss function
\begin{equation}
	MSE = \frac{1}{n}\sum_{i=1}^{n}(y_{i}-\hat{y}_i)^2,
\end{equation}
where ${y}_{i}$ and $\hat{{y}}_i$ are the  true values and the ANN predicted output values, respectively.
The sigmoid activation function, $\sigma(\zeta)$, is applied on all layers to return a normalized value of [0,1], represented mathematically as
\begin{equation}
	\sigma(\zeta) = \frac{1}{1 + e^{-\zeta}},
\end{equation}  
where $\zeta$ is a scalar value.  
\begin{table}
	\centering
	\caption{\label{sensitivity} Summary of ANN Hyperparameter Sensitivity Study}
	\begin{tabular}{ccccc}
     	\hline \hline
		Case & A & B & C & D \\ 
            \hline 
		Nodes & 5 & 10 & 20 & 40 \\ 
		Activation Function & sigmoid & tanh & relu & leakyrelu \\ 
            Batch Size & 25 & 50 & 100 & 200 \\ \hline
	\end{tabular}
\end{table}

A sensitivity study was performed to determine the training batch size, activation function for the hidden layers, and number of hidden layers to be used by the final model.  
Table \ref{sensitivity} summarizes the range of hyperparameters tested in this sensitivity study. 
Minimization of the mean-squared training and testing loss is used as the parameter to evaluate the ANN performance.   
Figure \ref{sensitivity_analysis} shows the resulting training loss $J_{\rm loss}$ for the ANN hyperparameter sensitivity analysis.  
The testing loss is not shown for brevity.  
The results show the most significant model improvement with increased hidden layer nodes and relative insensitivity to activation function or batch size.  
The final ANN model contains 20 hidden layer nodes, uses the sigmoid activation function, and trains on a batch size of 100. 
The ANN was trained for 100 epochs, which was found to be a good compromise between training and testing loss to ensure the model is not over- or under-fit.  
Training of the ANN model required approximately 2 minutes on a single Intel\textregistered$\;$ Core\texttrademark$\;$ i7-8550U CPU.

\begin{figure}[htbp]
    \centering
    \includegraphics[width=\columnwidth]{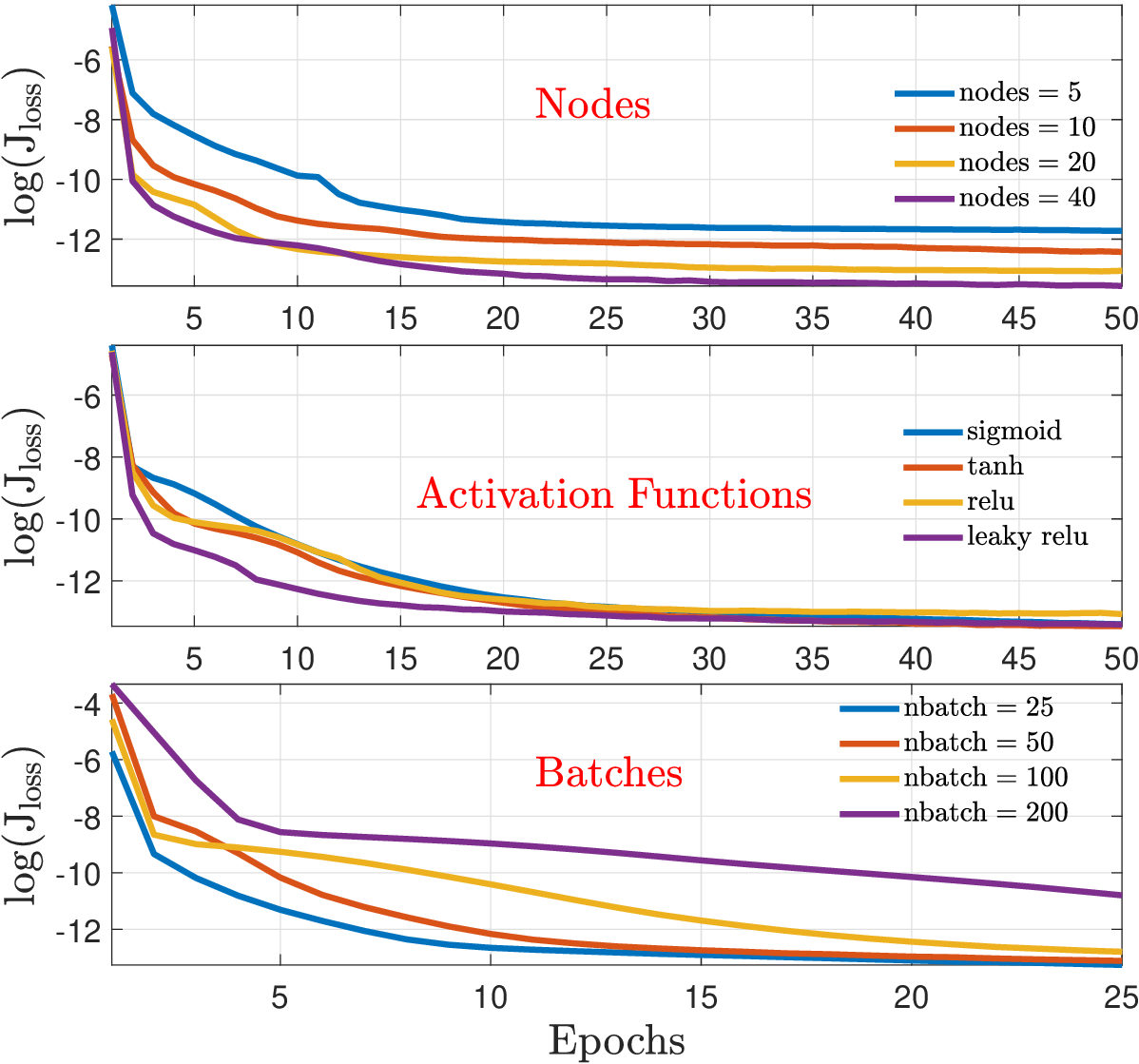}
    \caption{Resulting training loss of ANN hyperparameter sensitivity analysis.}
    \label{sensitivity_analysis}
\end{figure}


\section{Adaptive Controller}
\label{sec:control}
This section describes the learning-based controller used to control the thrust generated by the SFRJ model considered in this work.
The adaptive control architecture to regulate the SFRJ thrust is shown in Figure \ref{fig:BSL}.
The input to the computational SFRJ model is the freestream capture radius, as shown in Figure \ref{IntegratedSystemDiagram}.
In this work, we first consider the ideal scenario where the output is the thrust generated by the SFRJ. 
However, in practice, the thrust generated by the SFRJ can not be directly measured. 
Therefore, we then use in-situ combustor measurements and a neural network trained with open-loop simulations to predict the thrust generated by the SFRJ. 
We consider an adaptive PID controller whose gains are optimized by the retrospective cost adaptive control (RCAC) algorithm.
The output of the controller is the throat radius.
The controller gains are recursively updated using only the controller output's past values and the thrust measurement's past values. 

\begin{figure}[htbp]
    \centering
    \resizebox{\columnwidth}{!}
    {
    \begin{tikzpicture}[auto, node distance=2cm,>=latex',text centered]
    
        \node at (-3,0) (reference) {};
        \node[sum, right = 3 em of reference] (sum) {};
        \draw[->] (sum.east) -- +(0.5,0) -- +(0.5,-1) -- +(1.25,-1)
                    -- +(3,1);
        \node [smallblock, fill=green!20, right = 3 em of sum] (Controller) 
        {$\begin{array}{c}{\rm Adaptive } \\ {\rm Controller}\end{array}$};



        \node [smallblock, right = 1.5 em of Controller] (Plant) 
        {$\begin{array}{c}{\rm Computational} \\ {\rm SFRJ \ Model} \\ {(\rm Section \ \ref{sec:SFRJmodel})}\end{array}$};

        \node [smallblock, fill=green!20, right = 3 em of Plant] (ANN) 
        {$\begin{array}{c}{\rm Trained} \\ {\rm ANN}\end{array}$};
        
        

        
        \draw[->] (reference) node[xshift = 1.5em, yshift = 1.6em] {$\begin{array}{c}{\rm Commanded } \\ {\rm thrust}\end{array}$} -- node[xshift = 1.5em, yshift = -1.7em]{$-$} (sum);
        \draw[->] (sum) -- node[xshift = 0.25em, yshift = 0em]{$\begin{array}{c}{\rm Thrust } \\ {\rm error}\end{array}$}(Controller);
        \draw[->] (Controller)-- node[xshift = 0em, yshift = .2em]{$r_0$} (Plant);
        \draw[->] (Plant)-- node[xshift = 0em, yshift = .2em]{$\matl r_0 \\ P_{t4} \\ X_{CO} \\ H \matr$} (ANN);
        \draw[->] (ANN.east) -- node[xshift = 1em, yshift = 0em]{$\begin{array}{c}{\rm Predicted } \\ {\rm thrust}\end{array}$} +(1,0) --  +(+1,-1.5)
                    -| (sum.south);

    \end{tikzpicture}
    }
    \caption{Control architecture to regulate the thrust generated by the SFRJ with in situ measurements. }
    \label{fig:BSL}
\end{figure}
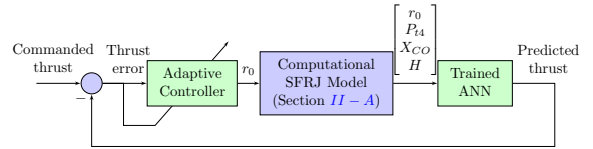

The adaptive PID controller is given by
\begin{align}
    u_k = K_{\rmP,k} z_k + K_{\rmI,k} \gamma_k + K_{\rmD,k} (z_k-z_{k-1}),
    \label{eq:PI_control}
\end{align}
where $z_k$ is the performance variable defined as the difference between the commanded thrust $r_k$ and the measured thrust output $y_k,$ that is, $z_k \isdef r_k - y_k,$
$\gamma_k$ is the integrated performance variable given by
\begin{align}
    \gamma_k
        \isdef 
            \sum_{i=0}^k z_i, 
\end{align}
and the scalars $K_{\rmP,k} $, $K_{\rmI,k}$ and $K_{\rmD,k}$ are the proportional, integral and derivative gains optimized by the RCAC algorithm at step $k. $
The integral signal $\gamma_k$ is computed recursively as
    $\gamma_{k+1} = \gamma_k + z_{k+1}.$
%
Note that the PID control law \eqref{eq:PI_control} can be written in the regressor form as
    $u_k = \Phi_k \theta_k, $
where 
\begin{align}
    \Phi_k 
        \isdef 
            \matl 
                z_k & \gamma_k & z_k-z_{k-1}
            \matr, 
    \quad
    \theta_k 
        \isdef 
            \matl 
                K_{\rmP,k} \\
                K_{\rmI,k} \\
                K_{\rmD,k}
            \matr,
\end{align}
where the regressor matrix $\Phi_k$ contains the measured data and the controller gain vector $\theta_k$ is optimized by the RCAC algorithm. To determine the controller gains $\theta_k$, let $\theta \in \BBR^{l_\theta}$, and consider the \textit{retrospective performance variable} defined by
\begin{align}
    \hat{z}_{k}(\theta)
        \isdef
            z_k+ 
            G_\rmf(\shiftq) (\Phi_{k} \theta - u_{k}),
    \label{eq:zhat_def}
\end{align}
where 
    $G_\rmf(\shiftq) 
        \isdef
            \sum_{i=1}^{n_\rmf} \frac{N_i}{\shiftq^i}.$
is a finite-impulse response filter. 
Note that $N_i \in \BBR^{l_z \times l_u}.$
%
Furthermore, define the \textit{retrospective cost function} $J_k \colon \BBR^{l_\theta} \to [0,\infty)$ by
\begin{align}
    J_k(\theta) 
        &\isdef
            \sum_{i=0}^k
                \hat{z}_{i}(\theta) ^\rmT 
                R_z
                \hat{z}_{i}(\theta)
                 +
                (\theta-\theta_0)^\rmT 
                P_0^{-1}
                (\theta-\theta_0),
    \label{eq:RetCost_def}
\end{align}
where $R_z \in \BBR^{l_z \times l_z}$, $ R_u \in \BBR^{l_u \times l_u},$ and $P_0\in\BBR^{l_\theta\times l_\theta}$ are positive definite; and $\theta_0\in\BBR^{l_\theta}$ is the initial vector of controller gains.
%

\begin{proposition}
    Consider \eqref{eq:RetCost_def}, 
    where $\theta_0 \in \BBR^{l_\theta}$ and $P_0 \in \BBR^{l_\theta \times l_\theta}$ is positive definite. 
    For all $k\ge0$, denote the minimizer of $J_k$ given by \eqref{eq:RetCost_def} by
    \begin{align}
        \theta_{k+1}
            \isdef
                \underset{ \theta \in \BBR^n  }{\operatorname{argmin}} \
                J_k({\theta}).
        \label{eq:theta_minimizer_def}
    \end{align}
    Then, for all $k\ge0$, $\theta_{k+1}$ is given by 
    \begin{align}
        \theta_{k+1} 
            &=
                \theta_k  - 
                 P_{k+1}\Phi_{\rmf, k}^\rmT R_z
                \left( z_k + \Phi_{\rmf,k} \theta_k - u_{\rmf,k} \right)
                 , \label{eq:theta_update}
    \end{align}
    where 
    \begin{align}
        P_{k+1} 
            &=
                P_{k}
                -  
                P_k  
                \Phi_{\rmf,k} ^\rmT 
                \left( 
                    R_z ^{-1} +  
                    \Phi_{\rmf,k}
                    P_k
                    \Phi_{\rmf,k} ^\rmT 
                \right)^{-1}
                \Phi_{\rmf,k} P_k,
        \label{eq:P_update_noInverse}
    \end{align}
    and
    $\Phi_{\rmf,k}
            \isdef 
                G_\rmf(\shiftq) \Phi_{k} $ and $ 
        u_{\rmf,k}
            \isdef 
                G_\rmf(\shiftq) u_{k}.$
\end{proposition}
\begin{proof}
    Due to page limits, a sketch of the proof is provided below. 
    Note that $J_k(\theta)$ is a quadratic function of $\theta,$ whose minimizer can be computed recursively using the RLS algorithm as shown in \cite{islam2019recursive}.
    $\theta_{k+1}$ and $P_{k+1}$ given by \eqref{eq:theta_update} and \eqref{eq:P_update_noInverse} are then the specialization of the RLS algorithm. 
\end{proof}

Finally, the control is given by
    $u_{k+1} = \Phi_{k+1} \theta_{k+1}.$

\section{Simulation Results}
This section presents several numerical examples demonstrating the application of the RCAC algorithm to synthesize a controller for command following using in situ measurements. 
In all of the command following examples, we set $P_0 = 10^{-5}I_3, N_1 = 1$. 
Note that $I_3$ is the $3\time3$ identity matrix.
These RCAC hyperparameters are obtained by a trivial grid search method to obtain a satisfactory transient response in a nominal scenario. 
%
In all subsequence examples, 
The capture radius $r_0,$ modulated by RCAC, is 
\begin{align}
    r_0 = \overline{r}_0 + 0.001 u_k, 
\end{align}
where $\overline r_0 = 53.58$ $\rm mm$ is the nominal capture radius.
Note that the factor of 0.001 multiplying $u_k$ ensures that the numerical values populating the regressor matrix $\phi$ have similar magnitudes, thus ensuring that the retrospective optimization problem is numerically well-posed.

\subsection{Command Following}
\subsubsection{Step Command}
The SFRJ is commanded to generate a constant command of $100$ $\rmN.$
Figure \ref{fig:single_step_NNB_no_noise} shows the closed-loop response with the learning PID controller in the loop, 
a) shows the commanded and the generated thrust,
b) shows the capture radius $r_0$ modulated by RCAC,
c) shows the absolute value of the tracking error $z$ on a log scale, 
d) shows the PID controller gains $\theta$ updated by RCAC at each step, 
e) shows the regression rate of the solid fuel, and 
f) shows the total pressure at station ``4''.  The total pressure is a useful proxy for system efficiency and is integral to calculating thrust, specific impulse, and characteristic velocity.
Note that RCAC optimizes the controller coefficients using only the measured data and does not rely on the SFRJ model to update the controller gains.

\label{sec:results}
\begin{figure}[htbp]
    \centering
    \includegraphics[width=\columnwidth]{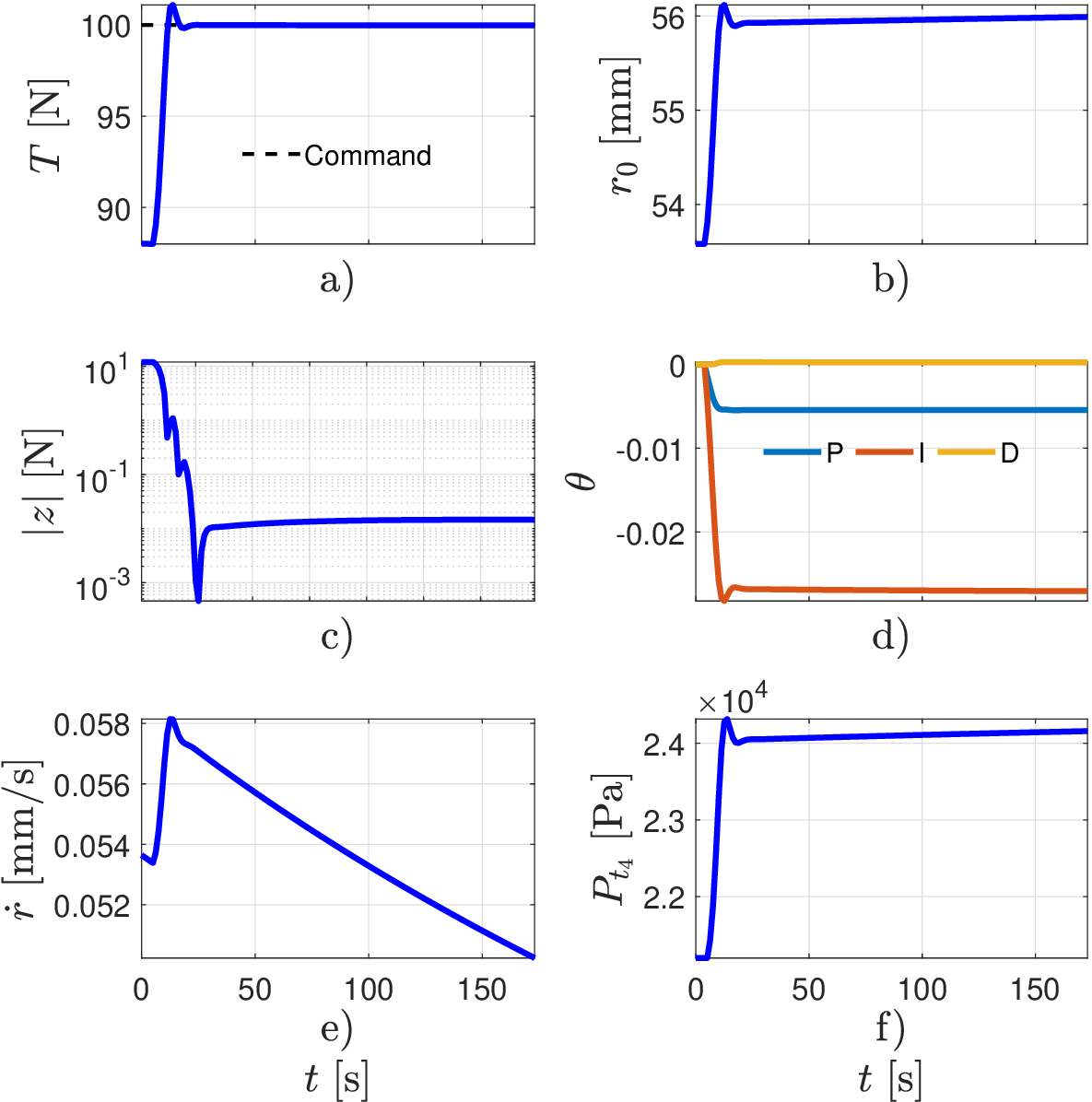}
    \caption{Closed-loop response of the SFRJ to a single step command.}
    \label{fig:single_step_NNB_no_noise}
\end{figure} 



\subsubsection{Doublet Command}
Next, we investigate the performance of the adaptive controller in the case where the SFRJ is commanded to follow a doublet command.
Note that RCAC hyperparameters are not changed.
Figure \ref{fig:square_NNB_no_noise} shows the closed-loop response of the SFRJ to a doublet command.

\begin{figure}[htbp]
    \centering
    \includegraphics[width=\columnwidth]{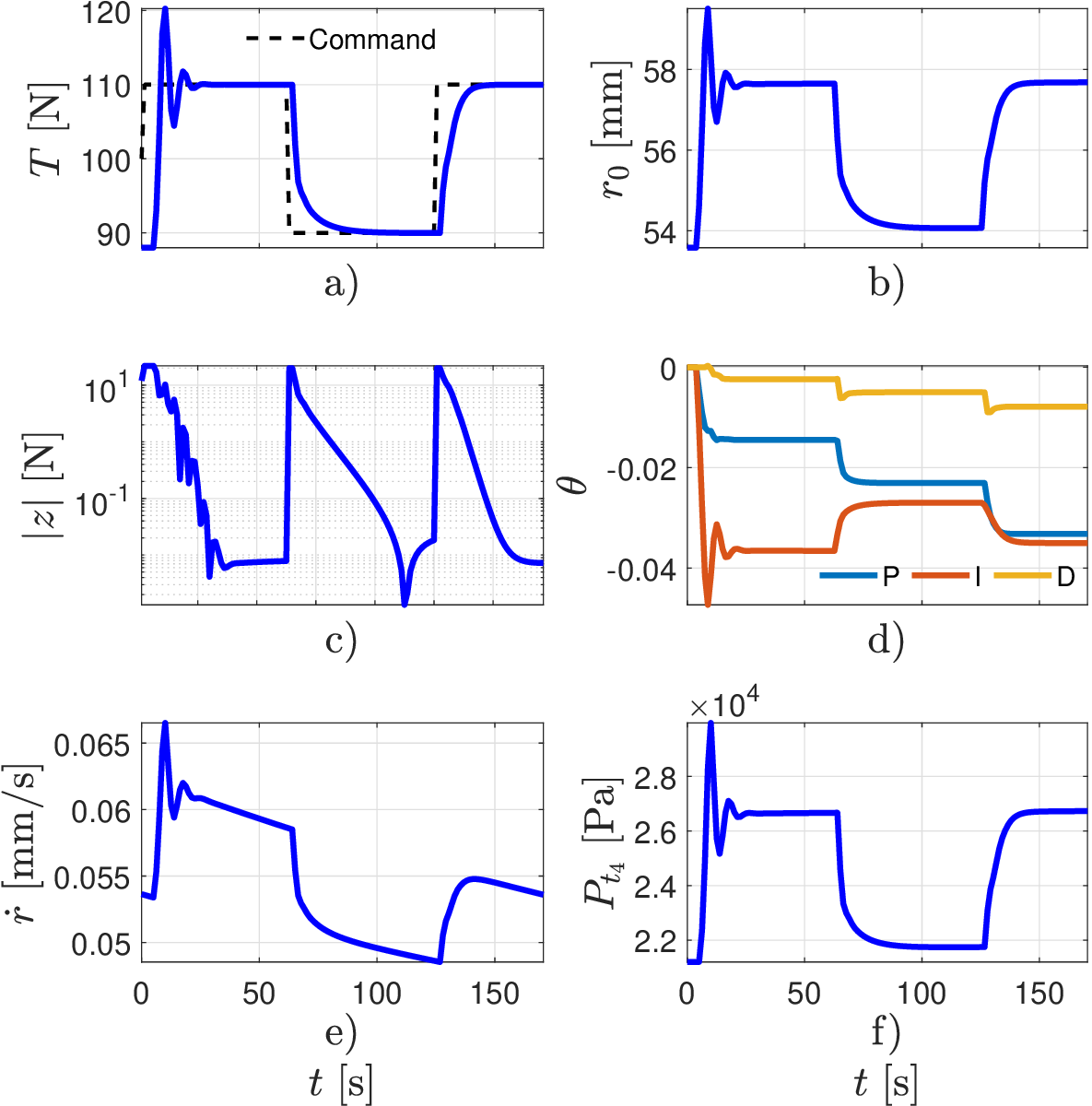}
    \caption{Closed-loop response of the SFRJ to a doublet command.}
    \label{fig:square_NNB_no_noise}
\end{figure}

\subsubsection{Ramp Command}
Next, we investigate the performance of the adaptive controller in the case where the SFRJ is commanded to follow a ramp command.
Note that RCAC hyperparameters are not changed.
Figure \ref{fig:ramp_up_down_NNB_no_noise} shows the closed-loop response of the SFRJ to a sequence of ramp commands.
\begin{figure}[htbp]
    \centering
    \includegraphics[width=\columnwidth]{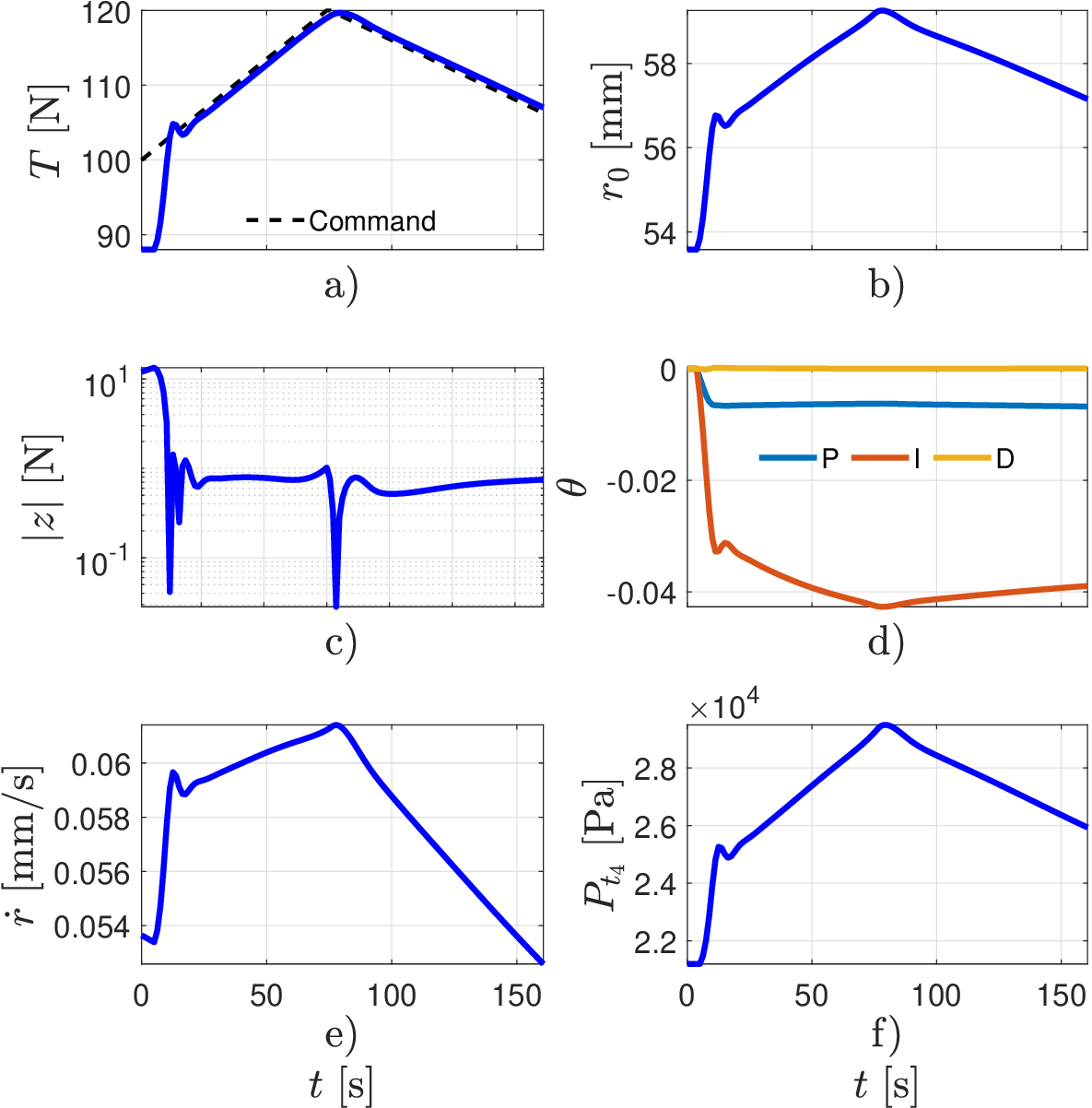}
    \caption{Closed-loop response of the SFRJ to a sequence of ramp commands.}
    \label{fig:ramp_up_down_NNB_no_noise}
\end{figure}

\subsubsection{Harmonic Command}
Next, we investigate the performance of the adaptive controller in the case where the SFRJ is commanded to follow a harmonic command.
Note that RCAC hyperparameters are not changed.
Figure \ref{fig:Sine_NNB_no_noise} shows the closed-loop response of the SFRJ to a harmonic command.

\begin{figure}[htbp]
    \centering
    \includegraphics[width=\columnwidth]{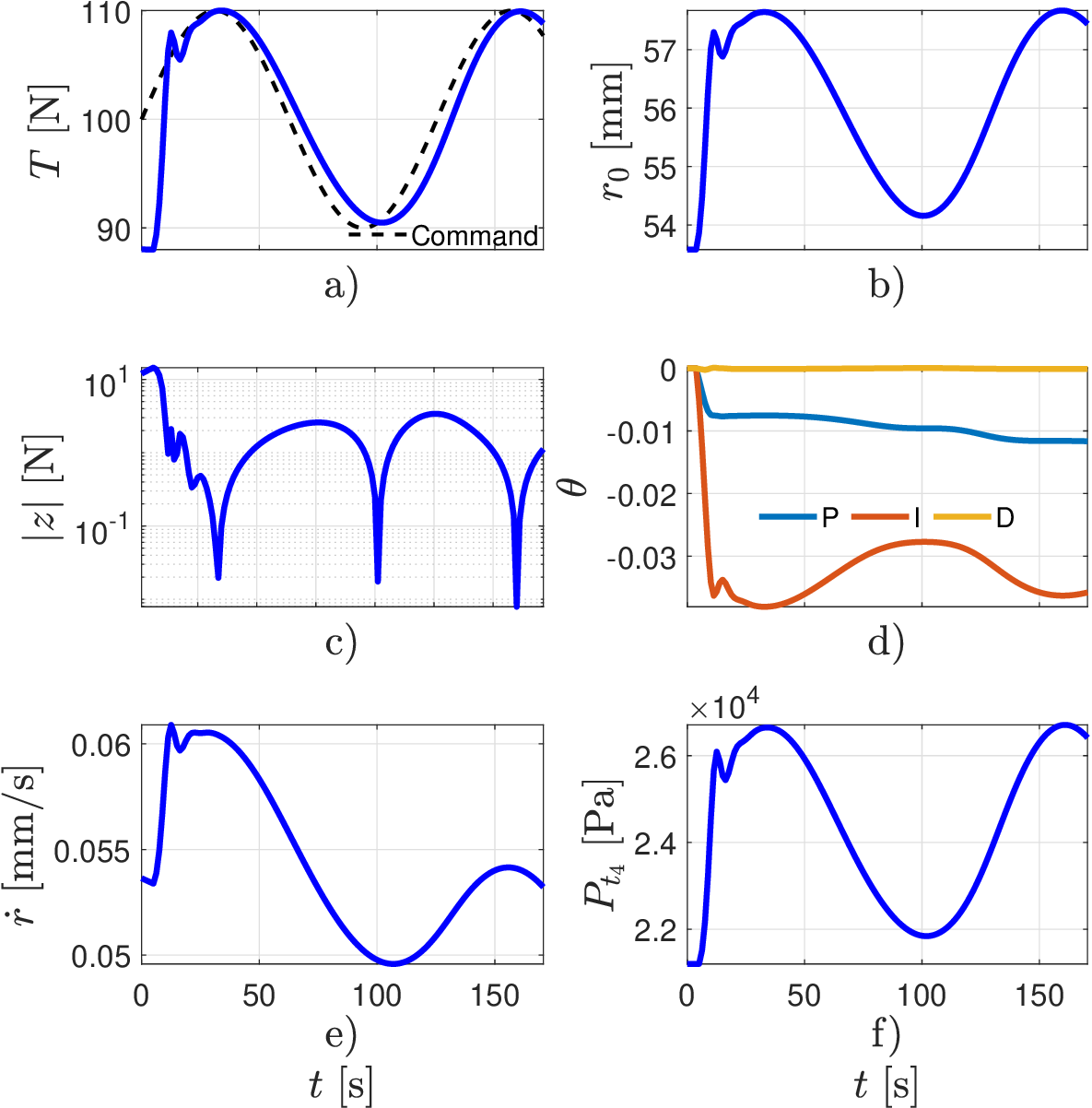}
    \caption{Closed-loop response of the SFRJ to a harmonic command.}
    \label{fig:Sine_NNB_no_noise}
\end{figure}

\subsection{Effect of Hyperparameters}
Next, we investigate the effect of RCAC hyperparameters on the closed-loop performance. 
We reconsider the step command-following problem.
In RCAC, we set $N_1 = n$ and $P_0 = p I_3,$ where $n = 0.1, 1, 10$ and $p = 10^{-4}, 10^{-5}, 10^{-6}, 10^{-7}.$
The closed-loop simulation is thus run twelve times with all combinations of $n$ and $p.$
Figure \ref{fig:RCAC_SFRJ_NNB_no_noise_Sensitivity} shows the effect of RCAC hyperparameters on the closed-loop response of the SFRJ. 
Note that a larger value of $P_0$ yields faster convergence but results in a larger overshoot. 
Similarly, a larger value of $N_1$ yields a faster response.

\begin{figure}[htbp]
    \centering
    \includegraphics[width=0.9\columnwidth]{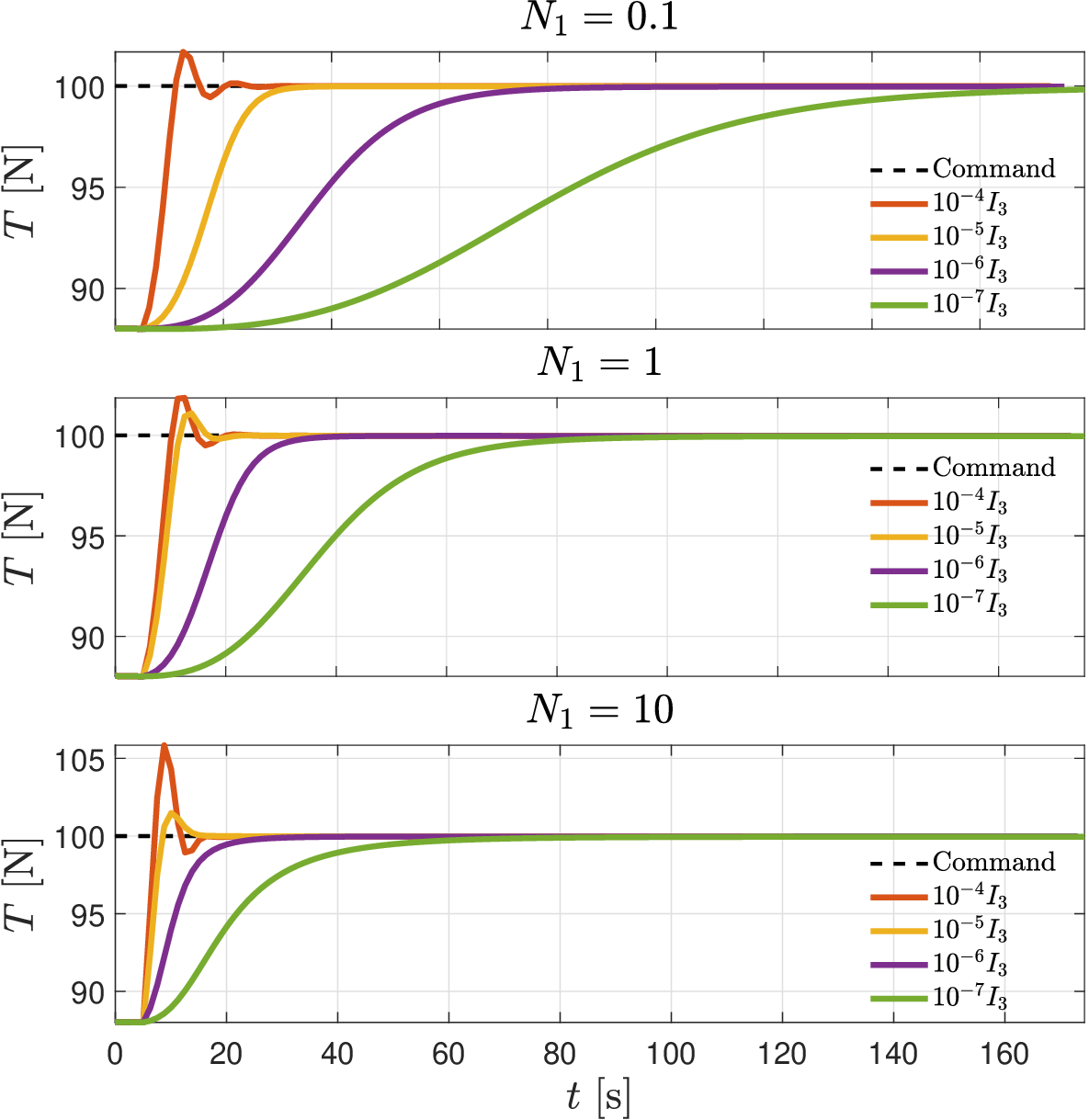}
    \caption{Effect of RCAC hyperparameters $P_0$ and $N_1$ on the closed-loop response of the SFRJ.}
    \label{fig:RCAC_SFRJ_NNB_no_noise_Sensitivity}
\end{figure}

\section{Conclusions}
\label{sec:conclusion}


The deployment of SFRJ systems is hindered by the difficulty of understanding and controlling the complex multi-physics combustion process and directly measuring the onboard-generated thrust.
This work described a framework for in-situ thrust monitoring and regulation using a neural network coupled with a learning-based adaptive controller.
A quasi-one-dimensional thermodynamic model with variable inlet geometry was used to model the SFRJ dynamics. 
A neural network was trained to predict thrust based on in-situ measurements of inlet geometry variable, altitude, combustor pressure, and exhaust $CO$ composition.
A sensitivity analysis was performed to determine the optimal neural network structure for minimizing the mean-squared error loss. 
With the neural network predictions as the proxy for thrust generated by the SFRJ, numerical simulations demonstrated successful command following response for several thrust commands without the need to retune the adaptive controller.  
Furthermore, the closed-loop response was found to be robust to variations in the adaptive controller's hyperparameters.



\section{Acknowledgment}
This research was supported by the Office of Naval Research grant N00014-23-1-2468.  RD was supported by the DoD National Defense Science and Engineering Graduate (NDSEG) Fellowship


\printbibliography

\end{document}